\documentclass[11pt]{article}
\usepackage{amssymb}
\usepackage[left=1in,top=1in,right=1in,bottom=1in]{geometry}
\usepackage{setspace}

\usepackage{amssymb, amsmath, amsthm, graphicx,mathrsfs}

\newtheorem{thm}{Theorem}[section]
\newtheorem{cor}[thm]{Corollary}

\newtheorem{lem}[thm]{Lemma}

\newtheorem{question}[thm]{Question}

\def\gammadelta{\gamma}
\def\eps{\varepsilon}
\begin{document}

\pagestyle{myheadings}
\markright{{\small{\sc Z.~F\"uredi} and {\sc R.~Luo}:  {\sc Monochromatic components in almost complete graphs}
}}

\title{\vspace{-0.5in}
  Large monochromatic components in almost complete graphs
  \\ and
  bipartite graphs}

\author{
{{Zolt\'an F\"uredi}}\thanks{
\footnotesize {Alfr\'ed R\'enyi Institute of Mathematics, Hungary.
E-mail:  \texttt{z-furedi@illinois.edu}. 
Research is supported in part by the Hungarian National Research, Development and Innovation Office grant NKFI-133819. 
}}
\and{{Ruth Luo}}\thanks{University of California, San Diego, La Jolla, CA 92093, USA. E-mail: {\tt ruluo@ucsd.edu}. Research of this author
is supported in part by NSF grant DMS-1902808.
}}
\date{\today}

\maketitle

\vspace{-0.3in}

\begin{abstract}
Gy\'arfas proved that every coloring of the edges of $K_n$ with $t+1$ colors contains a monochromatic connected component of size at least $n/t$. Later, Gy\'arf\'as and S\'ark\"ozy asked for which values of $\gammadelta=\gammadelta(t)$ does the following strengthening for {\em almost} complete graphs hold: if $G$ is an $n$-vertex graph with minimum degree at least $(1-\gammadelta)n$, then every $(t+1)$-edge coloring of $G$ contains a monochromatic component of size at least $n/t$. We show $\gammadelta
= 1/(6t^3)$
suffices, improving a result of DeBiasio, Krueger, and S\'ark\"ozy.

\medskip\noindent
{\bf{Mathematics Subject Classification:}} 05C55, 05C35.  
\end{abstract}

\section{Introduction, a stability of edge colorings} 
Erd\H{o}s and Rado 
observed
that every $2$-edge-coloring of the complete graph $K_n$ has a monochromatic spanning tree.
Generalizing this result, Gy\'arf\'as~\cite{gyarfas} proved that every $(t+1)$-edge-coloring of the edge set $E(K_n)$  
contains a monochromatic connected component of size at least $n/t$.
This bound is the 
best possible when $n$ is divisible by $t^2$ and an affine plane of order $t$ exists.

 Gy\'arf\'as and S\'ark\"ozy~\cite{GS2} 
proved that Gy\'arf\'as' theorem has a remarkable stability property, the complete graph $K_n$  can be replaced with graphs of high minimum degree.

\begin{question}[Gy\'arf\'as and S\'ark\"ozy~\cite{GS2}]\label{q1}
Let $t \geq 2$.
Which values of $\gammadelta=\gammadelta(t)$ guarantee that every $(t+1)$-edge-coloring of any $n$-vertex graph with minimum degree at least $(1-\gammadelta)n$ contains a monochromatic component of size at least $n/t$?
\end{question}

Let $\gammadelta(t)$ denote the best value we can have.
The case for $t=1$ is trivial, $\gammadelta(1)=0$.
It is observed in~\cite{GS1} that {\em any} non-complete graph has a $2$-edge-coloring without a monochromatic spanning tree: if $xy$ is a non-edge, consider any edge-coloring where every edge incident to $x$ is red and every edge incident to $y$ is blue. Then there does not exist monochromatic component containing both $x$ and $y$.

The case for at least three colors (i.e., $t\geq 2$) is more interesting.
Gy\'arf\'as and S\'ark\"ozy~\cite{GS2}  showed $\gammadelta
 \leq 1/(1000t^9)$ suffices. This was
 improved to
 $1/(3072 t^5)$  by DeBiasio, Krueger, and S\'ark\"ozy~\cite{DKS}.

It was also conjectured in~\cite{GS2} that $\gammadelta(t)$ could be as big as $t/(t+1)^2$. This was disproved for $t=2$ by Guggiari and Scott~\cite{guggscott} and by Rahimi~\cite{rahimi}, and more recently for general $t$ by DeBiasio and Krueger~\cite{DK}. The constructions of graphs in~\cite{DK,guggscott, rahimi} are based on modified affine planes. They have minimum degree
at least
$(1-\frac{t-1}{t(t+1)})n-2$ and a $(t+1)$-edge coloring in which
each monochromatic component is of order less than $n/t$.

DeBiasio, Krueger, and S\'ark\"ozy~\cite{DKS} proposed a version for bipartite graphs.

\begin{question}[DeBiasio, Krueger, and S\'ark\"ozy~\cite{DKS}]\label{q2}Let $t \geq 2$ and $n_1\leq n_2$. Determine for which values of $\gammadelta = \gammadelta(t, n_1, n_2)$ the following is true: let $G$ be an $X_1,X_2$-bipartite graph such that $|X_i| = n_i$ for $i \in \{1,2\}$,  for every $x \in X_1$, $d(x)\geq (1 - \gammadelta)n_2$, and for every $y \in X_2$, $d(y) \geq (1 - \gammadelta)n_1$. Then every $t$-edge-coloring of $G$ contains a monochromatic component of order at least $n/t$.
\end{question}

They proved that $\gammadelta(t, n_1, n_2)\leq
   (n_1/n_2)^3/(128t^5)$ suffices.
For both Questions~\ref{q1} and~\ref{q2} the $t=2$ case is solved completely in~\cite{guggscott,rahimi} and~\cite{DKS}, respectively.
They obtained $\gammadelta(2)=1/6$, $\gammadelta(1, n_1, n_2)\geq 1/2$, and $\gammadelta(2, n_1, n_2)\geq 1/3$ (independently of $n$), and these constants are the best possible.
So from now on, we only consider $t\geq 3$.

Our main result is an improvement for the bound on $\gammadelta(t, n_1, n_2)$ in Question~\ref{q2} which in turn implies a better bound for $\gammadelta(t)$ in Question~\ref{q1}.

\begin{thm}\label{main2}Fix integers $t \geq 3$, $n_1, n_2$ such that $n_2 \geq n_1\geq 1$  
  and let $\gammadelta\leq \dfrac{(n_1/n_2)}{t^3}$.
Let $G$ be an $X_1,X_2$-bipartite graph such that $|X_i| = n_i$ for $i \in \{1,2\}$, \smallskip \newline \smallskip
\centerline{for every $x \in X_1$, $d(x)\geq (1 - \gammadelta)n_2$, and for every $y \in X_2$, $d(y) \geq (1 - \gammadelta)n_1$.}
 Then every $t$-edge-coloring of $G$ contains a monochromatic component of order at least $n/t$.
\end{thm}

\begin{cor}\label{main}Fix integers $n, t \geq 3$,  
 and let $\gammadelta\leq
     1/(6t^3)$.
Suppose $G$ is an $n$-vertex graph with minimum degree at least $(1 - \gammadelta)n$. Then any coloring of $E(G)$ with $t+1$ colors contains a monochromatic connected component with at least $n/t$ vertices.
\end{cor}

Our method is very similar to that in~\cite{GS2} or in~\cite{DKS}.
The major difference is that
 we will first collect a series of inequalities in the next section, apparently unrelated to graphs and graph colorings. But
using these tight inequalities, we prove Theorem~\ref{main2} in Section 3 and Corollary~\ref{main} in Section 4.

We use standard notation.
E.g., $[s]$ stands for the set of integers $\{ 1,2, \dots, s \}$.
The {\em degree} of a vertex $v$ in $G$ is denoted $d_G(v)$ or simply  $d(v)$ when there is no room for ambiguity.

\section{Inequalities}

\begin{lem}\label{I1}
Let $a_1, \ldots, a_s, b_1, \ldots, b_s, E, M, A, B$ be non-negative real numbers such that \newline
--- $\sum_{i=1}^s a_ib_i \geq E$,\newline
--- for all $i \in[s]$, $a_i+b_i \leq M$, \newline
--- $\sum_{i=1}^s a_i \leq A$, \;\; and \;\; $\sum_{i=1}^s b_i \leq B$.\newline
Then $E(A+B) \leq M A B$.
\end{lem}

\begin{proof}
The case $EAB=0$ is easy, so we may suppose $A, B, E>0$.
Apply
 Jensen's inequality for the convex function $x^2$
\[\left( \frac{ \sum_{i=1}^s b_i a_i}{\sum_{i=1}^s b_i}\right)^2 \leq \frac{\sum_{i=1}^s b_i a_i^2}{ \sum_{i=1}^s b_i}. \]
Therefore
\[ \frac{(\sum_{i=1}^s a_i b_i)^2}{\sum_{i=1}^s b_i} \leq \sum_{i=1}^s a_i^2 b_i,\] and similarly $\frac{(\sum_{i=1}^s a_i b_i)^2}{\sum_{i=1}^s a_i} \leq \sum_{i=1}^s a_i b_i^2$.
So we have
\begin{eqnarray*}
E \sum_{i=1}^s a_ib_i \left(\frac{1}{A} + \frac{1}{B} \right) & \leq & \left(\sum_{i=1}^s a_i b_i \right)^2 \left(\frac{1}{\sum_{i=1}^s a_i} + \frac{1}{\sum_{i=1}^s b_i}\right)\\
& \leq & \sum_{i=1}^s a_i^2 b_i + a_ib_i^2 \;\; = \;\; \sum_{i=1}^s (a_ib_i) (a_i + b_i)
\leq M \sum_{i=1}^s a_ib_i.
\end{eqnarray*}

Dividing by $(\sum_{i=1}^s a_ib_i)$ and simplifying, we have $E(A^{-1} + B^{-1}) = E(A + B)/(AB) \leq M$.
\end{proof}

\begin{lem}\label{I5}Fix $n_1, n_2, t, a_1, \ldots, a_s, b_1, \ldots, b_s \geq 0$, $\eps \geq 0$.
Suppose $t>1$, 
$n_1, n_2 >0$, \newline
---
$\sum_{i =1}^s a_ib_i \geq (1 - \eps) \frac{n_1n_2}{t}$,\newline
---
$\sum_{i=1}^s a_i \leq n_1$, \;\; $\sum_{i=1}^s b_i \leq n_2$, and 
\newline
--- $a_i+ b_i < (n_1+n_2)/t$ for all $i \in [s]$.
Then for all $i \in [s]$,

\begin{equation}\label{Lemma2.2}
  a_i < \frac{n_1}{t} + \frac{\sqrt{\eps (t-1) n_1n_2}}{t}\;\;\text{and} \;\; b_i < \frac{n_2}{t} + \frac{\sqrt{\eps (t-1) n_1n_2}}{t}.
\end{equation}
\end{lem}

\begin{proof}
We prove the statement only for $a_1$, as the proofs for other $a_i$'s and $b_i$'s are symmetric.

First, we handle the case $a_1=n_1$.
Then $a_2=\dots=a_s=0$ so the first constraint gives $a_1b_1=n_1b_1\geq (1 - \eps) \frac{n_1n_2}{t}$.
Hence $(1 - \eps)n_2/t \leq b_1$. Combining this with the last constraint we get
\[    n_1 + (n_2/t)  - (\eps n_2)/t \leq a_1+b_1< (n_1/t)+(n_2/t).
  \]
Rearranging we have $(t-1)n_1<   \sqrt{\eps (t-1) n_1n_2}$.
So the value of the upper bound for $a_1$ in~\eqref{Lemma2.2} exceeds $n_1$, so the inequality holds.

Second, consider the case $b_1=n_2$. Then  the last constraint implies $a_1< (n_1+n_2)/t -b_1 =(n_1+n_2)/t -n_2< n_1/t$, so~\eqref{Lemma2.2} holds.
From now on, we may suppose that $n_1-a_1$ and $n_2-b_1$ are both positive.

Third, suppose that $\sum_{i=2}^s a_ib_i \geq \frac{(n_1 -a_1) (n_2 - b_1)}{t-1}$.
Let $M:= \max_{2\leq i \leq s} \{a_i + b_i\}$, $A = n_1 - a_1$, $B= n_2 - b_1$. Then by Lemma~\ref{I1}, we obtain
\[\frac{(n_1 -a_1) (n_2 - b_1)}{t-1} (n_1 - a_1 + n_2 - b_1) \leq M (n_1 -a_1)(n_2 - b_1).\]
Simplify by the positive term $(n_1-a_1)(n_2-b_1)$
 \[M \geq \frac{n_1 - a_1 + n_2 - b_1}{t-1} \geq \frac{n_1 + n_2 - (n_1 + n_2)/t}{t-1} 
 = \frac{n_1 + n_2}{t},\] a contradiction.

Therefore, in the last case we consider, we may assume
\[\frac{(n_1 - a_1) (n_2 - b_1)}{t-1} + a_1b_1 > \sum_{i=1}^s a_i b_i \geq (1 - \eps)\frac{n_1n_2}{t}.\]
Rearranging, we get
\begin{eqnarray*}
(n_1 -a_1) (n_2 - b_1) + (t-1)(a_1b_1) & > & (1-\eps)\frac{(t-1)(n_1n_2)}{t}\\\Rightarrow \enskip n_1n_2 - n_1 b_1 - n_2 a_1 + t a_1b_1 & > & n_1n_2 - \frac{n_1n_2}{t} -\eps \frac{(t-1)n_1n_2}{t}\\
\Rightarrow \enskip \frac{n_1n_2}{t} + \eps\frac{(t-1)n_1n_2}{t} & > & n_2 a_1 -b_1(ta_1 -n_1).
\end{eqnarray*}
If $a_1 < n_1/t$, then we are done. So assume $a_1 \geq n_1/t$ (so $ta_1 - n_1\geq 0$). We add the non-positive term $(a_1+ b_1 - (n_1 + n_2)/t)(ta_1 - n_1)$ to the right hand side to obtain
\begin{eqnarray*}
 \frac{n_1n_2}{t} + \eps\frac{(t-1)n_1n_2}{t} & > & n_2 a_1 -b_1(ta_1 -n_1) + (a_1+ b_1 - \frac{n_1 + n_2}{t})(ta_1 - n_1)\\
 & = & n_2a_1 + ta_1^2 - a_1n_1 -n_1a_1 + \frac{n_1^2}{t} -n_2a_1 + \frac{n_1n_2}{t}\\
 \Rightarrow \enskip 0 & > & ta_1^2 - 2n_1 a_1 + \left( \frac{n_1^2}{t} - \eps \frac{(t-1)n_1n_2}{t}\right)
\end{eqnarray*}
Solving for $a_1$, we obtain
\[a_1 < \frac{2n_1 + \sqrt{4n_1^2 - 4(n_1^2 -\eps (t-1)n_1n_2)}}{2t} = \frac{n_1 + \sqrt{\eps(t-1)n_1n_2}}{t}.\]
\end{proof}

\begin{lem}\label{I2}
Fix $\eps \geq 0$, integers $1 \leq t \leq s$, and reals $a_1, \ldots, a_s, b_1, \ldots, b_s \geq 0$ such that\newline
---
$a_1 \geq \ldots \geq a_s \geq 0$, \;\;
\newline ---
$\sum_{i=1}^s a_i = n_1$,\;\;
$\sum_{i=1}^s b_i = n_2>0$,\;\;
\newline ---
for all $i \in [s]$, $a_i +b_i \leq (n_1+n_2)/t$, \;\;
\newline ---
$\sum_{i=1}^s a_i b_i \geq  (1-\eps)n_1n_2 /t$ .
\newline
Let $a := a_{t+1} + \ldots + a_s$. 
Then \[a \leq \eps n_1 \frac{n_1 + n_2}{n_2}.\] In particular, if $n_1 \leq n_2$, then $a \leq 2\eps n_1$.
\end{lem}

\begin{proof}
 We construct a new sequence $b_1', \ldots, b_s'$ with $b_i' \geq b_i$ for $i \in [t]$, $b_j'=0$ for $t< j\leq s$,  such that $\sum_{i=1}^t b_i' = \sum_{i=1}^s b_i =n_2$, and $a_i + b_i' \leq (n_1+ n_2)/t=: M$ for all $i \in [t]$.
Note that these conditions together with the fact that the $a_i$'s are non-increasing imply that $\sum_{i=1}^t a_ib_i' \geq \sum_{i=1}^s a_ib_i$.

We build our sequence greedily starting with $b_1, \ldots, b_s$.
Define a set $I\subseteq [s]$ as follows
\smallskip\newline \smallskip \centerline{$I(b_1, \dots, b_s): = \{i\in [t], a_i+b_i< M \} \cup \{j: j> t,\,  b_j>0 $\}.
}
If for all $j \geq t+1$, $b_j=0$, then we let $b_1', \ldots, b_s' = b_1, \ldots, b_s$ and we are done.
So suppose some $j \geq t+1$ satisfies $b_j \neq 0$.
Then there exists $i \in [t]$ with $b_i + a_i < M$ because $\sum_{i=1}^t (a_i+b_i)\leq n_1+n_2-b_j= tM -b_j$.
If $a_i+b_i+b_j\leq M$ then we update $b_i' = b_i + b_j$, $b_j' = 0$ and otherwise $b_k'=b_k$ for all $k \in [s]\setminus\{i,j\}$.
If $a_i+b_i+b_j > M$ then we update $b_i' = b_i + M - (a_i + b_i) = M-a_i$, $b_j' = b_j - (M - (a_i + b_i))$ and $b_k'=b_k$ for $k\in [s]\setminus \{ i,j\}$.
In both cases we get $I(b_1', \dots, b_s') \subsetneq I(b_1, \dots, b_s)$, so one can continue this process at most $s$ steps until we get
$I(b_1', \dots, b_s') \subset [t]$.

So suppose we have found a sequence $b_1', \ldots, b_t'$ as desired. Apply Lemma~\ref{I1} on the sequences $a_1, \ldots, a_t$ and $b_1', \ldots, b_t'$. We have $\sum_{i=1}^t a_i= n_1 - a=:A$, $\sum_{i=1}^t b_i' =n_2=: B$, $\sum_{i=1}^t a_ib_i' \geq \sum_{i=1}^s a_ib_i \geq (1-\eps)n_1n_2 /t =: E$, and $a_i + b_i' \leq M$ for all $i \in [t]$. Therefore,
 \[\frac{(1- \eps)n_1n_2}{t} (n_1 + n_2 - a) \leq \frac{n_1 + n_2}{t} (n_1 - a)n_2\]

 Rearranging and solving for $a$, we get
 \begin{eqnarray*}
a (n_2 + \eps n_1) & \leq & \eps n_1^2 + \eps n_1 n_2 \\
\Rightarrow\enskip  a & \leq & \eps n_1 \frac{n_1 + n_2}{n_2 + \eps n_1 } \leq \eps n_1 \frac{n_1 + n_2}{n_2}.
 \end{eqnarray*}
\end{proof}

\section{Proof of Theorem~\ref{main2} for almost complete bipartite graphs}

\begin{proof}
Let $G$ be an $X_1,X_2$-bipartite graph with $|X_1| = n_1$, 
 $|X_2| = n_2,$ and $n_2 \geq n_1\geq 1$.
Consider any coloring of the edges of $G$ with colors $1, \ldots, t$.
For a color $i \in [t]$, we denote by $G^i$ the spanning subgraph of edges colored with $i$. Suppose that every monochromatic component has less than $(n_1+n_2)/t$ vertices.
We claim that $|E(G^i)| < n_1n_2/t$. 
Indeed, 
 let $C_1, \ldots, C_s$ be the connected components of  $G^i$. For $j \in [s]$, let $a_j = |C_j \cap X_1|$, $b_j = |C_j \cap X_2|$. Then $E:=|E(G^i)| \leq \sum_{j=1}^s a_jb_j$. Apply Lemma~\ref{I1} with $A = n_1$, $B = n_2$, $M = (n_1+n_2-1)/t$. We get
   \[E \leq (n_1+n_2-1)/t \cdot (n_1 + n_2)^{-1} \cdot(n_1n_2) < n_1n_2/t,\]
as desired.

Let $\eps_i$ be such that $|E(G^i)| = (1-\eps_i)n_1n_2/t$.
By Lemma~\ref{I5}, 
 a connected component of color $i$ contains at most $\frac{n_\alpha}{t} + \frac{\sqrt{\eps_i(t-1)n_1n_2}}{t}$ vertices from $X_\alpha$, $\alpha \in \{1,2\}$. Therefore, for any $i\in [t]$, $x \in X_1$ and $y \in X_2$,
\begin{equation}\label{eq4}
d_{G^i}(x) < \frac{n_2}{t} + \frac{\sqrt{\eps_i(t-1)n_1n_2}}{t}, \qquad
d_{G^i}(y) < \frac{n_1}{t} + \frac{\sqrt{\eps_i(t-1)n_1n_2}}{t} .
\end{equation}

Since $|E(G)| \geq (1-\gammadelta)n_1n_2$, we have $\sum_{i=1}^t \eps_i \leq t\gammadelta$.
Without loss of generality, suppose color $1$ satisfies $\eps_1 \leq \gammadelta$. Let $C_1, \ldots, C_s$ be the vertex sets of the connected components of color $1$, ordered so that $|X_1\cap C_1 | \geq \ldots\geq  |X_1 \cap C_s|$. Define $a_j, b_j$ as before. Note that $s \geq t+1$, since the $C_j$'s cover $V(G)$ and $|C_j| < (n_1 + n_2)/t$ for all $j$.
By Lemma~\ref{I2}, $a:= a_{t+1} + \ldots + a_s \leq 2\eps_1 n_1$.

{\bf Case 1}: $X_2 \cap (C_{t+1} \cup \ldots \cup C_s) \neq \emptyset$. Fix a vertex $y$ in this set. Then $d(G^1)(y) \leq 2 \eps_1 n_1$. We get
\begin{eqnarray*}
(1-\gammadelta)n_1 \;\;  \leq \;\; d_G(y) &< &  2\eps_1 n_1 + \frac{n_1(t-1)}{t} + \sum_{i=2}^t \frac{\sqrt{\eps_i(t-1)n_1n_2}}{t}\\
& \leq & 2\gammadelta n_1 + n_1 - \frac{n_1}{t} + \frac{\sqrt{(t-1)^2(\sum_{i=2}^t \eps_i)n_1n_2}}{t}\\
& \leq & 2\gammadelta n_1 + n_1 - \frac{n_1}{t} + \sqrt{\gammadelta t n_1n_2} \cdot \frac{t-1}{t}.\\
\end{eqnarray*}
Here we used the fact that $\sum_{i=2}^t \frac{\sqrt{\eps_i}}{t-1} \leq \sqrt{\frac{\sum_{i=2}^t \eps_i}{t-1}}$ because 
$\sqrt x$ is a concave function. Therefore
\[\frac{n_1}{t} < n_13\gammadelta + \sqrt{\gammadelta t n_1n_2} \cdot \frac{t-1}{t} \leq n_1 3\frac{(n_1/n_2)}{t^3} + \sqrt{t\frac{(n_1/n_2)}{t^3}n_1n_2} \cdot \frac{t-1}{t} \leq \frac{n_1}{t}\left( \frac{3}{t^2} + \frac{t-1}{t}\right),\]
a contradiction when $t \geq 3$.

{\bf Case 2}: $X_2 \cap (C_{t+1} \cup \ldots \cup C_s) = \emptyset$. Let $x \in X_1 \cap (C_{t+1} \cup \ldots \cup C_s)$. By the case, $x$ is not incident to an edge of color $1$. So we instead obtain
\begin{eqnarray*}
(1-\gammadelta)n_2 \;\; \leq \;\; d_G(x) &< & \frac{n_2(t-1)}{t} + \sum_{i=2}^t \frac{\sqrt{\eps_i(t-1)n_1n_2}}{t}\\
& \leq &  n_2 - \frac{n_2}{t} + \sqrt{\gammadelta t n_1n_2}\cdot \frac{t-1}{t}.\\
\end{eqnarray*}
This implies that \[\frac{n_2}{t} < n_2 \gammadelta + \sqrt{\gammadelta t n_1n_2}\cdot \frac{t-1}{t} \leq n_2 \frac{(n_1/n_2)}{t^3} + \sqrt{t\frac{(n_1/n_2)}{t^3}n_1n_2}\cdot \frac{t-1}{t} \leq \frac{n_1}{t} \left( \frac{1}{t^2} + \frac{t-1}{t}\right),\]
a contradiction since $n_1 \leq n_2$ and $t \geq 3$.
\end{proof}

\section{Proof of Corollary~\ref{main} for almost complete graphs}

\begin{proof}
Let $G$ be an $n$-vertex graph with minimum degree at least $(1- \gammadelta) n$, and suppose the edges of $G$ are colored with colors $0, 1, \ldots, t$ such that each monochromatic connected component has size less than $n/t$. Again, we use $G^i$ to refer to the spanning subgraph of the edges of color $i$.

%

Let $V_1, \ldots, V_r$ be the vertex sets of the connected components of $G^0$. We will split the vertex set into two almost equal parts $X_1$ and $X_2$ such that the size of each part is in the range $[n(\frac{1}{2} -\frac{1}{2t}), n(\frac{1}{2} + \frac{1}{2t})]$, and each set $V_i$ is contained either entirely in $X_1$ or entirely in $X_2$. To see that this is possible, arbitrarily add entire sets $V_i$ to $X_1$ until $|X_1| < n(\frac{1}{2} +\frac{1}{2t})$ but adding any additional set to $X_1$ causes the size of $X_1$ to be at least
$n(\frac{1}{2} +\frac{1}{2t})$. Then let $X_2 = V(G) - X_1$. At this point, $|X_1| >  n(\frac{1}{2} -\frac{1}{2t})$, otherwise all sets $V_j$ not contained in $X_1$ have size at least $n/t$,
a contradiction.

Now let $|X_1| = n_1$, $|X_2| = n_2$, where without loss of generality, $|X_1| \leq |X_2| < 2|X_1|$ (and $n= n_1 + n_2$). By construction, there are no edges of color $0$ between $X_1$ and $X_2$. Hence, the edges of the bipartite subgraph $G[X_1,X_2]$ are colored with $t$ colors.
(Here $G[X,Y]$ denotes the spanning bipartite subgraph of $G$ in which we include only edges with endpoints in both $X$ and $Y$.)

For simplicity, set $G' = G[X_1,X_2]$. Let $x \in X_1$ and $y \in X_2$. Then \[d_{G'}(x) \geq n_2 - \gammadelta n = n_2 - \gammadelta(n_1 + n_2) \geq (1- 2\gammadelta)n_2,\] and \[d_{G'}(y) \geq n_1 - \gammadelta n = n_1 - \gammadelta(n_1 + n_2) \geq n_1 - \gammadelta(n_1 + 2n_1) = (1-3\gammadelta)n_1.\]

 Since $G'$ does not have a monochromatic component of size at least $n/t = (n_1 + n_2)/t$, Theorem~\ref{main2} implies that \[3\gammadelta \geq \frac{(n_1/n_2)}{t^3} > \frac{1/2}{t^3} = \frac{1}{2t^3}.\]

We get a contradiction when $\gammadelta \leq 1/(6t^3)$.

\end{proof}


\end{document}